\documentclass[10pt]{amsart}

\usepackage{amsthm}
\usepackage{amssymb}
\usepackage{amsmath}
\usepackage{mathtools}
\usepackage{pdfpages}
\usepackage{float}
\usepackage[colorlinks=true,pdftex,unicode=true,linktocpage,bookmarksopen,hypertexnames=false]{hyperref}
\usepackage{tikz-cd}
\usepackage{caption}
\usepackage[capitalise]{cleveref}
\usepackage{physics}
\usepackage{quiver}
\usepackage{MnSymbol}
\usepackage{graphicx}
\usepackage{relsize}

\newcommand{\Aut}{\operatorname{Aut}}

\newcommand{\End}{\operatorname{End}}

\DeclareMathOperator{\eva}{ev}
\DeclareMathOperator{\diag}{diag}

\makeatletter
\numberwithin{equation}{section}
\numberwithin{figure}{section}
\numberwithin{table}{section}
\newtheorem{thm}{Theorem}[section]
\newtheorem*{thm*}{Theorem}
\newtheorem{lemma}[thm]{Lemma}
\newtheorem{corollary}[thm]{Corollary}
\newtheorem{proposition}[thm]{Proposition}
\newtheorem{notation}[thm]{Notation}

\theoremstyle{definition} 
\newtheorem{definition}[thm]{Definition}

\newtheorem{remark}[thm]{Remark}
\newtheorem{example}[thm]{Example}

\DeclareMathOperator{\Sym}{Sym}
\DeclareMathOperator{\inv}{inv}

\makeatother

\title{Free commutative skew braces}
\author{T.Letourmy }

\address[T. Letourmy]{D\'epartement de Math\'ematique, Universit\'e Libre de Bruxelles, Boulevard du Triomphe, B-1050 Brussels, Belgium; and 
Department of Mathematics and Data Science, Vrije Universiteit Brussel, Pleinlaan 2, 1050 Brussel, Belgium}
\email{thomas.letourmy@ulb.be}


\begin{document}
\begin{abstract}
   The main result of this paper is an explicit construction of the free commutative skew brace---that is, a skew brace whose circle group is commutative---on an arbitrary generating set \(X\). We embed this object into a set of rational functions and show that a simple linear equation characterizes the image of this embedding. As a consequence, comparing elements in this skew brace is no more difficult than comparing elements in the free commutative group generated by \(X\). 
\end{abstract}
\maketitle
\section{Introduction}
The Yang–Baxter equation (YBE)~\cite{baxter2016exactly}, originally arising in the context of statistical mechanics and quantum field theory, plays now a key role in quantum groups, braided category, and low-dimensional topology.

A particularly elegant and combinatorially rich class of solutions to the YBE are the set-theoretic solutions, introduced by Drinfeld in 1992 in a seminal list of open problems~\cite{10.1007/BFb0101175}. These are solutions defined on a set $X$ via a bijective map $r\colon X\times X\to X\times X$ satisfying the braid relation
\[
(r \times \mathrm{id}) \circ (\mathrm{id} \times r) \circ (r \times \mathrm{id}) = (\mathrm{id} \times r) \circ (r \times \mathrm{id}) \circ (\mathrm{id} \times r)
\]
on \( X \times X \times X \).

In~\cite{RUMP2007153}, Rump showed that the involutive, non degenerate, set theoretic solutions can be studied through an algebraic structure he called \emph{braces}. Later Guarnieri and Vendramin generalized braces to \emph{skew braces} to study the non-involutive solutions~\cite{guarnieri2017skew}. Skew braces now have applications in other areas of mathematics such as Hopf--Galois structures and Lie rings ~\cite{MR3291816,MR3763907}. 

A skew brace is a triple $(B,\star, \circ)$\footnote{In the literature it is more conventional to use a $+$ operation instead of the $\star$. We made this choice to avoid confusion whenever we consider simultaneously brace type objects and rings.} such that $(B,\star)$ and $(B,\circ)$ are groups and the compatibility condition $a\circ (b\star c)=(a\circ b)\star a^{-1}\star (a\circ c)$ holds for all $a,b,c\in B$ where $a^{-1}$ denotes the inverse of $a$ with respect to $\star$. As skew braces form a variety (in the sense of universal algebra), the free skew brace on any set $X$ exists. However, the existence proof does not provide an explicit construction. Since explicitly constructing free skew braces appears to be difficult, it is natural to begin by studying certain subvarieties with a simpler structure. If both operations are commutative, the structure reduces to that of a \emph{commutative radical ring}. The free commutative radical ring generated by any set $X$ is well known and admits an explicit construction based on fractions of polynomials~\cite{kepka2007commutative}. A construction of the free one-generated centrally nilpotent left brace with right
nilpotency class two and a given left nilpotency class was given in~\cite{BALLESTER–BOLINCHES_ESTEBAN–ROMERO_KURDACHENKO_PÉREZ–CALABUIG_2025}. In this paper, we advance the study of free skew braces by giving an explicit construction of the free commutative skew brace---that is, a skew brace with commutative circle group---on an arbitrary generating set $X$, see Theorem~\ref{theorem:freeskb}. To this end, we consider a subvariety of skew-rings (in the sense of universal algebra; see~\cite{MR3977806}), which we call \emph{wires}. A wire is essentially a two-sided skew brace---i.e., a skew brace that also satisfies right brace distributivity---except that the operation $\circ$ is, in general, only a monoid. In our terminology, skew-rings correspond to \emph{right wires}. In a way, wires are to two-sided skew braces what rings are to division rings. For instance, we present a construction of the skew brace of fractions (Proposition~\ref{pro:fractions}), where we show that one can construct a skew brace from a commutative wire by taking fractions. This is a crucial result of the paper, because the skew brace of fractions has a universal property (Proposition~\ref{pro:universalproperty}) that reduces the problem of finding an explicit construction of the free commutative skew braces to that of finding one of the free commutative wires. 

We present several families of wires constructed from rings (see Proposition~\ref{pro:projection}, and Examples~\ref{ex:wiresringmorph} and~\ref{ex:wiresother}). Additionally, we provide a construction of the \emph{free wire} over an arbitrary set \(X\). This construction resembles that of the polynomial ring except that the addition is replaced by a non commutative operation and the usual distributivity is replaced by the brace distributivity. 

The \emph{free commutative wire}, denoted \(F_{\mathrm{cw}}(X)\), is obtained from the free wire by imposing commutativity on the circle operation \(\circ\). Using Fox calculus—a theory developed in the 1950s by Ralph Fox, originally motivated by knot theory~\cite{MR53938}—we show that comparing elements in \(F_{\mathrm{cw}}(X)\) reduces to comparing elements in \(\mathbb{Z}^{(X)}\), the free commutative group generated by \(X\). Moreover, this approach enables us to embed \(F_{\mathrm{cw}}(X)\) into the polynomial ring \((\mathbb{Z}[\mathbb{Z}^{(X)}])[(t_x)_{x \in X}]\), where the coefficients lie in the group algebra over \(\mathbb{Z}^{(X)}\).

Surprisingly, as shown in Theorem~\ref{pro:derivativeproduct}, this embedding is a monoid morphism from \((F_{\mathrm{cw}}(X), \circ)\) to the multiplicative monoid of the ring \((\mathbb{Z}[\mathbb{Z}^{(X)}])[(t_x)_{x \in X}]\). Since the target is an integral domain, it follows that \((F_{\mathrm{cw}}(X), \circ)\) is cancellative. This ensures that the passage to the skew brace of fractions introduces no additional relations. Consequently, comparisons in the free commutative skew brace over \(X\) can still be reduced to comparisons in \(\mathbb{Z}^{(X)}\).

The paper is organized into four sections. Section~\ref{sec:wires} recalls the construction of the free commutative radical rings, introduces the basics of wires, presents examples, and describes the skew brace of fractions. Section~\ref{sec:freewire} is dedicated to the construction of the free wires. In Section~\ref{sec:freeskb}, we construct the free commutative wire on an arbitrary set $X$ and show how Fox calculus can be used to compare its elements. The section concludes with a description of the free commutative skew brace over $X$.

\section{Wires}
\label{sec:wires}
\begin{notation}
    In the paper we will often use the notations $\star$ and $\circ$ for some binary operations.
    When it exists, the inverse of an element $a$ with respect to $\star$ will be denoted by $a^{-1}$, whereas its inverse with respect to $\circ$ will be written $a'$. If $\star$ has a neutral element we will denote it by $e$.
    
    The operation $\circ$ has higher precedence than the operation $\star$.

    When working in a group $(G,\star)$, for $g,h\in G$ we will denote the conjugation by $g^h=h^{-1}\star g\star h$. By $[g,h]$ we mean the commutator $g\star h\star g^{-1}\star h^{-1}$. 
    
    If $N_1,N_2$ are two normal subgroups of $G$ we will denote by $[N_1,N_2]$ the subgroup generated by the elements $[n_1,n_2]$ with $n_1\in N_1$ and $n_2\in N_2$.

    When working in a quotient object, we will denote by $[u]$ the equivalence class of some element $u$.
\end{notation}

\begin{definition}
    A skew brace is a triple $(B,\star,\circ)$ such that $(B,\star)$ and $(B,\circ)$ are groups and
    \begin{equation}
    \label{eq:leftdistrib}
        a\circ (b\star c)=a\circ b\star a^{-1}\star a\circ c
    \end{equation}
    holds for all $a,b,c\in B$. A skew brace is called \emph{two-sided} if
    \begin{equation}
    \label{eq:rightdistrib}
        (b\star c)\circ a=b\circ a\star a^{-1}\star c\circ a
    \end{equation}
    for all $a,b,c\in B$. If $\circ$ is commutative, we will call $(B,\star,\circ)$ a \emph{commutative skew brace}. Dually, a \emph{brace} is a skew brace such that $\star$ is commutative. The equations~\eqref{eq:leftdistrib} and~\eqref{eq:rightdistrib} will be called respectively \emph{left and right brace distributivity}.
\end{definition}
\begin{remark}
    \label{rem:trusses}
    Let $(B,\star)$ be a group and $(B,\circ)$ a monoid. Let $[b,c,d]=b\star c^{-1}\star d$ for all $a,b,c\in B$. Then, one can see that Equation~\eqref{eq:leftdistrib} is true for every triple of elements of $B$ if and only if \[
    a\circ [b,c,d]=[a\circ b,a\circ c,a\circ d]\]
    for all $a,b,c,d\in B$. Similarly, Equation~\ref{eq:rightdistrib} 
    is true for every triple of elements of $B$ if and only if \[
    [b,c,d]\circ a=[b\circ a,c\circ a,d\circ a]\]
    for all $a,b,c,d\in B$. This is a reformulation in the language of skew trusses introduced in~\cite{MR4443751}. This point of view will often be useful when doing computations.
\end{remark}
\begin{example}
    The simplest examples of skew braces are built from groups. Let $(G,\cdot)$ be a group, then $(G,\cdot,\cdot)$ and $(G,\cdot,\cdot^{\mathrm{op}})$ are skew braces where $g\cdot^{\mathrm{op}} h=h\cdot g$ for all $g,h\in G$. The former is called a \emph{trivial skew brace} and the latter an \emph{almost trivial skew brace}.
\end{example}
\begin{definition}
    A radical ring is a non unital ring $(R,+,\cdot)$ such that the operation $x\circ y=x+xy+y$ makes $(R,\circ)$ a group.
\end{definition}
With this operation, one can check that $(R,+,\circ)$ is a two-sided brace. In fact radical rings are exactly the two-sided braces~\cite{Rump}.
\begin{proposition}
    Let $(B,\star,\circ)$ be a two-sided brace. Then, with $+=\star$ and $a\cdot b=a^{-1}\star a\circ b\star b^{-1}$ for all $a,b\in B$, the triple $(B,+,\cdot)$ is a radical ring.
\end{proposition}
 The free commutative radical rings are well known and can be found in~\cite{kepka2007commutative}.
\begin{proposition}
\label{pro:freeradicalring}
   Let $X$ be a set, $R=\mathbb{Z}[X]$ the polynomial ring with variables in $X$ and denote by $Q$ the field of fractions of $R$. Let $I(X)$ be the ideal generated by $X$ in R. The free commutative radical ring generated by $X$ is the subring $F_R(X)=\left\{\frac{f}{1+g}: f,g\in I(X)\right\}$ of $Q$.
\end{proposition}

\begin{remark}
\label{rem:discfreecombr}
As any non unital ring, a radical ring $(R,+,\cdot)$ can be embedded in a unital ring $(S,+,\cdot)$, its Dorroh extension. In $S$ one can twist the addition 
\begin{equation}
\label{eq:twist}
x\oplus y=x-1+y
\end{equation}
so that the original multiplication of the ring verifies the brace distributivities. In fact, there is an isomorphism of skew braces $(R,+,\circ) \cong (R+1,\oplus,\cdot)$ where $R+1=\{r+1: r\in R\}$ (Proposition 2.9 of~\cite{kepka2007commutative}). Through this isomorphism, the operation $\circ$ corresponds to the product $\cdot$ of the ring which is easier to understand.
    For example, there is an isomorphism of skew braces 
    \[(F_R(X),+,\circ)\cong \left(\left\{\frac{f}{g}: f,g\in I(X)+1\right\},\oplus, \cdot\right).\]
Another useful perspective is to consider the subset \( R(X) \subseteq \mathbb{Z}[X]\), consisting of polynomials whose coefficients sum to one. Via the evaluation map $x\mapsto x+1$ for all $x\in X$, one obtains an isomorphism
\[(F_R(X),+,\circ)\cong \left(\left\{\frac{f}{g}: f,g\in R(X)\right\},\oplus, \cdot\right).\]
 $R(X)$ is exactly the set of polynomials which can be written as words in the monomials of $\mathbb{Z}[X]$. In other words, $f\in R(X)$ if and only if it can be written as $m_1^{\oplus\epsilon_1}\oplus \dots \oplus m_n^{\oplus\epsilon_n}$, for some $\epsilon_i\in \{1,-1\}$ and $m_i$ monomials in $X$, where $m^{\oplus -1}$ denotes the inverse of $m$ with respect to $\oplus$. This viewpoint is particularly natural when working with free objects, and it forms the basis for our generalization to construct the free commutative skew brace on $X$.
\end{remark}
With these identifications, we can see that the free commutative brace on a set $X$ can be obtained by taking fractions of elements of $(I(X)+1,\oplus, \cdot)$ or $(R(X),\oplus,\cdot)$. This construction can be extended to more general objects.
\begin{definition}
    Let $(W,\star,\circ)$ be such that $(W,\star)$ is a group and $(W,\circ)$ is a semigroup. $(W,\star,\circ)$ is a left \emph{wire} if it follows left brace distributivity, it is a right wire if it follows right brace distributivity. If it follows both left and right brace distributivity, then $(W,\star,\circ)$ will be called a wire. A wire $W$ is called a \emph{commutative wire} if $(W,\circ)$ is commutative.
\end{definition}

\begin{remark}
    As in the case of skew braces, equations~\eqref{eq:leftdistrib} and~\eqref{eq:rightdistrib} both imply that $(W,\circ)$ is a monoid and that the two operations have the same neutral element.
\end{remark}
\begin{remark}
    $(W,\star,\circ)$ is a left wire if and only if $(W,\star,\circ^{\mathrm{op}})$ is a right wire, where the operation $\circ^{\mathrm{op}}$ is defined by $u\circ^{\mathrm{op}}v=v\circ u$.
\end{remark}
\begin{notation}
    We will denote by $W_\star$ the group $(W,\star)$ and $W_\circ$ the monoid $(W,\circ)$.
\end{notation}
\begin{remark}
    Such objects have already been considered in the literature. In~\cite{MR3977806} right wires were introduced under the name of skew-rings. In~\cite{MR4443751}, it is the left wires that are called skew-rings. We decided to put emphasis on the two-sided objects and give it a more concise name. the idea behind the name is that a wire is close to being a two-sided skew brace but it lacks some elements. This will be emphasized by Proposition~\ref{pro:fractions}. Since \( W_\circ \) only needs to be a semigroup, the name \emph{semibrace} could have been a natural choice. However, this term already appears in the literature for an other type of structures~\cite{CATINO2017163}.
\end{remark}
\begin{definition}
    Let $W$ and $V$ be left or right wires. A wire morphism is a map $f\colon W\to V$ such that $f(u\star v)=f(u)\star f(v)$ and $f(u\circ v)=f(u)\circ f(v)$ for all $u,v\in W$. If a wire morphism is bijective, then it will be called an isomorphism.
\end{definition}
\begin{example}
\label{ex:ringwire}
    Let $(R,+,\cdot)$ be a ring, then with $x\circ y=x+x\cdot y+y$ the triple $(R,+,\circ)$ is a wire. Conversely, if $(W,\star,\circ)$ is a wire with a commutative $\star$ operation, then $(W,\star,\cdot)$ with $u\cdot v=u^{-1}\star u\circ v\star v^{-1}$, is a ring.
\end{example}
\begin{remark}
\label{rem:ababfreewire}
    Example~\ref{ex:ringwire} shows that wires with commutative $\star$ operation are equivalent to (non necessarily unital) rings. Using the notations of Remark~\ref{rem:discfreecombr}, the object $I(X)$ is the free commutative ring generated by $X$. Hence, the wire $(I(X),+,\circ)$ is the free wire with both operations commutative generated by $X$ and it is isomorphic to the wires $(I(X)+1,\oplus, \cdot)$ and $(R(X),\oplus,\cdot)$.
\end{remark}
\begin{proposition}
\label{pro:projection}
    Let \((A, +, \cdot)\) be a unital ring, and let \(p, \pi \in \operatorname{End}((A, \cdot))\) be a monoid endomorphisms such that $p(u)\pi(v) = \pi(v)p(u)$ for all $u, v \in A$.
\begin{enumerate}
    \item Define a binary operation \(\star \colon A \times A \to A\) by
\[
u \star v := \pi(v)u - p(u)\pi(v) + p(u)v.
\]
Suppose that the following identities hold for all \(u, v \in A\):
\begin{equation}
\label{eq:cond1}
p(u \star v) = p(u)p(v), \quad \pi(u \star v) = \pi(v)\pi(u).
\end{equation}
Then \((A, \star)\) is a monoid.

\item Furthermore, let $G$ and $H$ be subgroups of $A^\star$ and
$U := p^{-1}(G) \cap \pi^{-1}(H)$,
and for any \(u \in U\), define
\[
\inv_{\star}(u) := \pi(u)^{-1} - p(u)^{-1}\pi(u)^{-1}u + p(u)^{-1}.
\]
If in addition to~\eqref{eq:cond1}, the following conditions are satisfied for all \(u \in U\):
\[
p(\inv_{\star}(u)) = p(u)^{-1}, \quad \pi(\inv_{\star}(u)) = \pi(u)^{-1},
\]
then \((U, \star, \cdot)\) forms a \emph{right wire}.\footnote{This proposition was established with the help of Victoria Lebed as a generalization of two examples appearing in a previous version.}
\end{enumerate}
\end{proposition}
    \begin{proof}
    \begin{enumerate}
    \item 
    First, we show that if we assume~\eqref{eq:cond1}, then $\star$ is associative. Let $u,v,w\in A$,
        \begin{align*}
    u\star(v\star w)=& \pi(w)\pi(v)(u-p(u))+p(u)(v\star w)\\
    =& \pi(w)(\pi(v)u-p(u)\pi(v)+p(u)v)-p(u)p(v)\pi(w)+ p(u)p(v)w\\
    =&(u\star v)\star w.
    \end{align*}
    Recall that a monoid endomorphism fixes the unit by definition, so the neutral element is $1$. 
    \item With the assumption, one can see that $U$ is stable under the operations $\star$, $\inv_\star$ and $\cdot$ and that the inverse of an element $u\in U$ with respect to the operation $\star$ is $\inv_{\star}(u)$.
    
    In addition, let $u,v,w\in U$,
    \begin{align*}
        [u,v,w]=&u\star \inv_\star(v)\star w\\
        =& \pi(w)\pi(v)^{-1}u-p(u)p(v)^{-1}\pi(w)\pi(v)^{-1}v+p(u)p(v)^{-1}w
    \end{align*}
    which implies that, $[u,v,w]x=[ux,vx,wx]$
    for all $x\in A$. This concludes the proof by Remark~\ref{rem:trusses}.\qedhere
    \end{enumerate}
    \end{proof}

    \begin{example}
    \label{ex:wiresringmorph}
    Using notations of Proposition~\ref{pro:projection}, let $p$ and $\pi$ be idempotent ring endomorphisms such that $p(u)\pi(v) = \pi(v)p(u)$ for all $u, v \in A$. Then, for every subgroups $G$ and $H$ of $A^*$, $(p^{-1}(G) \cap \pi^{-1}(H),\star,\cdot)$ is a right wire.   
    
    Prototypal examples are given by augmented algebras. Let $R$ and $S$ be commutative rings and $A$ an $R$-algebra which is also an $S$-algebra. Let $\theta\colon A\to R$ be an $R$-algebra morphism and $\tau\colon A\to S$ be an $S$-algebra morphism, then we can choose $p=\iota_R\theta$ and $\pi= \iota_S\tau$ where $\iota_R$ and $\iota_S$ are respectively the inclusion of $R$ and $S$ in the center of $A$. We will see that the free commutative wires can be obtained in this way (see the discussion that follows Definition~\ref{def:asspoly}). 
\end{example}
\begin{example}
\label{ex:wiresother}
    Other examples do not fall in the family presented in Example~\ref{ex:wiresringmorph}.
    For example if $p$ and $\pi$ are both the trivial monoid morphism---that is, the map sending every elements to one---then regardless of the groups $G$ and $H$ we end up with the wire $(A,\oplus,\cdot)$ obtained by twisting the addition as in~\eqref{eq:twist}.
    Another remarkable case is when $\pi$ is the trivial monoid morphism and $p$ an idempotent ring endomorphism or the other way around. 
    There are also non trivial examples in which $p$ and $\pi$ are both not ring morphisms. For example, let $R$ be a unital ring and $A$ the ring of upper triangular $n\times n$ matrices with coefficient in $R$. Let $E$ be a subset of $\{1,\dots,n\}$, we set $\pi$ to be the trivial monoid morphism and we define $p$ as

    \begin{equation*}
    p\left(
    \begin{pmatrix}
a_{11} & a_{12} & \cdots & a_{1n} \\
    & a_{22} & \cdots & a_{2n} \\
 &  & \ddots & \vdots \\
0      &       &  & a_{nn}
\end{pmatrix}
\right)=\diag(\mathbf{1}_E(1)(a_{11}-1)+1,\dots,\mathbf{1}_E(n)(a_{nn}-1)+1)\\
\end{equation*}
Where $\mathbf{1}_E$ is the indicator function. This amounts to saying that $p$ sends the upper triangular matrix with $(i,j)$ entry $a_{ij}$ to the diagonal matrix whose $i$-th diagonal entry is $1$ if $i\not\in E$ and $a_{ii}$ otherwise. Then, for any subgroups $G$ of $A^*$, the triple $(p^{-1}(G),\star,\cdot)$ is a right wire.
\end{example}
Our next goal is to construct skew braces from commutative wires. To that end, we begin by recalling how a group can be constructed from a commutative monoid.
\begin{definition}
    Let $(M,\circ)$ be a commutative monoid. On the cartesian product $M\times M$ define the equivalence relation $(u,v)\sim (u_1,v_1)$ if and only if there exists a $k\in M$ such that $u\circ v_1\circ k=u_1\circ v\circ k$. Then, the set 
    \[G(M)= (M\times M)/\sim\]
     is a group with operation $\frac{u}{v}\circ \frac{u_1}{v_1}=\frac{u\circ u_1}{v\circ v_1}$, where $\frac{x}{y}$ denotes the equivalence class of $(x,y)$. This group is called the Grothendieck group of $M$.
\end{definition}
There is a canonical map $\iota \colon M\to G(M)$ that sends $m$ to $\frac{m}{e}$ where $e$ is the neutral element of the monoid.
\begin{proposition}
    Let $(M,\circ)$ be a commutative monoid and $(A,\circ)$ a commutative group. Let $f\colon M\to A$ be a monoid morphism. Then, there exists a unique group morphism $\hat{f}\colon G(M)\to A$ such that $\hat{f}\iota=f$.
\end{proposition}
\begin{proof}
    The morphism $\hat{f}$ is defined by $\hat{f}(\frac{u}{v})=f(u)\circ f(v)'$.
\end{proof}
\begin{remark}
    The map $\iota$ is injective if and only if $M$ is cancellative. For example if a commutative monoid $M$ contains an element $0$ such that $0\circ u=0$ for all $u\in M$ then $G(M)$ is the trivial group. This is the case for the monoid of a unital commutative ring.
\end{remark}
\begin{proposition}
\label{pro:fractions}
 Let $(W,\star,\circ)$ be a commutative wire, and let $(B(W),\circ)$ denote the Grothendieck group of the monoid $W_\circ$. On $B(W)$, define the operation
 \[\frac{u}{v}\star\frac{u_1}{v_1}= \frac{u\circ v_1\star (v\circ v_1)^{-1}\star u_1\circ v}{v\circ v_1}.\]
  Then, $(B(W),\star,\circ)$ is a commutative skew brace. We call $B(W)$ the skew brace of fractions of $W$.
\end{proposition}
\begin{proof}
    First, we show that $\star$ is well-defined. Suppose that $\frac{x}{y}=\frac{u}{v}$ and $\frac{x_1}{y_1}=\frac{u_1}{v_1}$, we will show that $\frac{x}{y}\star \frac{x_1}{y_1}=\frac{u}{v}\star \frac{u_1}{v_1}$. 
    We know that there exist $k,k_1\in W$ such that 
    \begin{equation}
    \label{eq:noncanc}
         x\circ v\circ k=y\circ u\circ k \quad \text{and}\quad x_1\circ v_1\circ k_1= y_1\circ u_1\circ k_1.
    \end{equation}
    Let $c=v\circ v_1\circ k\circ k_1$ and $d= y\circ y_1\circ k\circ k_1$. Recall from Remark~\ref{rem:trusses} that,        
    \[(u\circ v_1\star (v\circ v_1^{-1})\star u_1\circ v)\circ d= u\circ v_1\circ d\star (v\circ v_1\circ d)^{-1}\star u_1\circ v\circ d.
    \]
    We deduce from~\eqref{eq:noncanc} that $u\circ v_1\circ d=x\circ y_1\circ c$, $v\circ v_1\circ d=y\circ y_1\circ c$ and $u_1\circ v\circ d=x_1\circ y\circ c$. Again using Remark~\ref{rem:trusses}, we obtain
    
    \[(u\circ v_1\star (v\circ v_1)^{-1}\star u_1\circ v)\circ d=(x\circ y_1\star (y\circ y_1)^{-1}\star x_1\circ y)\circ c.\]
    
    Let us show now that $\star$ is associative. Let $u,v,u_1,v_1,u_2,v_2\in W$. Then,
    \begin{align*}
        &\left(\frac{u}{v}\star \frac{u_1}{v_1}\right)\star \frac{u_2}{v_2} \\
        = &\frac{u\circ v_1\star (v\circ v_1)^{-1} \star u_1\circ v}{v\circ v_1}\star\frac{u_2}{v_2}\\
        = &\frac{u\circ v_1\circ v_2\star (v\circ v_1\circ v_2)^{-1} \star u_1\circ v\circ v_2 \star (v\circ v_1\circ v_2)^{-1} \star u_2\circ v\circ v_1}{v\circ v_1\circ v_2}\\
        = & \frac{u}{v}\star \frac{u_1\circ v_2\star (v_1\circ v_2)^{-1}\star u_2\circ v_1}{v_1\circ v_2}\\
        = & \frac{u}{v}\star \left(\frac{u_1}{v_1}\star\frac{u_2}{v_2}\right).
    \end{align*}
    Thus $(B(W),\star)$ is a group with neutral element $\frac{e}{e}$ and the inverse of $\frac{u}{v}$ is $\frac{v\star u^{-1}\star v}{v}$. 

    It is left to check that we have the brace distributivity. Using notations of Remark~\ref{rem:trusses},
    \[
    \left[\frac{u}{v},\frac{u_1}{v_1},\frac{u_2}{v_2}\right]=\frac{[u\circ v_1 \circ v_2, v\circ u_1\circ  v_2, v\circ v_1\circ u_2]}{v\circ v_1\circ v_2}
    \]
    for all $u,u_1,u_2,v,v_1,v_2\in W$.
    Then,
    \begin{align*}
        \left[\frac{x\circ u}{y\circ v},\frac{x\circ u_1}{y\circ v_1},\frac{x\circ u_2}{y\circ v_2}\right]=&\frac{y\circ y\circ[x\circ u\circ v_1\circ v_2, v\circ x\circ  u_1\circ v_2, x\circ v\circ v_1\circ u_2]}{y\circ v\circ y\circ v_1\circ y\circ  v_2}\\
        =& \frac{[x\circ u\circ v_1\circ  v_2, x\circ v\circ u_1\circ  v_2, x\circ v\circ v_1\circ u_2]}{y\circ v\circ v_1\circ v_2}\\
        =&\frac{x}{y}\circ \left[\frac{u}{v},\frac{u_1}{v_1},\frac{u_2}{v_2}\right]
    \end{align*}
     for all $x,y,u,u_1,u_2,v,v_1,v_2\in W$.
\end{proof}
\begin{remark}
    One can see that the free commutative brace over $X$ is in fact the skew brace of fractions of $(R(X),\oplus, \cdot)$. 
\end{remark}
\begin{proposition}
    \label{pro:universalproperty} 
Let $W$ be a commutative wire, $B$ a skew brace, and $f \colon W \to B$ a wire morphism. Then the canonical map $\iota \colon W \to B(W)$, which sends each $w \in W$ to $\frac{w}{e}$, is a wire morphism. Moreover, there exists a unique skew brace morphism $\hat{f} \colon B(W) \to B$ such that the following diagram commutes:
\[
\begin{tikzcd}
	B(W) \arrow[dashed]{r}{\hat{f}} & B \\
	W \arrow{u}{\iota} \arrow{ur}[swap]{f}
\end{tikzcd}
\]
\end{proposition}
\begin{proof}
    By definition, $\iota$ is a wire morphism. The map $\hat{f}$ is the one from the universal property of the Grothendieck group. We verify that this is a group morphism from $B(W)_\star$ to $B_\star$. Let $\frac{u}{v},\frac{u_1}{v_1}\in B(W)$, we have 
    \begin{align*}
    \hat{f}\left(\frac{u}{v}\star\frac{u_1}{v_1}\right)=&(u\circ v_1\star (v\circ v_1)^{-1}\star u_1\circ v)\circ(v\circ v_1)'\\
    =& u\circ v'\star u_1\circ v_1'\\
    =& \hat{f}\left(\frac{u}{v}\right)\star\hat{f}\left(\frac{u_1}{v_1}\right).\qedhere
    \end{align*}
\end{proof}
As this will be useful in the sequel, we develop the rudiments of the theory of wires, following the already well established theory of skew braces.
\begin{proposition}
\label{rem:actions}
    If $W$ is a left wire, there is a canonical left action by group \emph{endomorphisms} 
\begin{equation*}
    \lambda\colon W_\circ\to \End(W_\star),\quad \lambda_u(v)=u^{-1}\star u\circ v.\\
\end{equation*}
    If $W$ is a right wire, there is a canonical right action by group endomorphisms
    \begin{equation*}
        \rho\colon W_\circ\to \End(W_\star), \quad\rho_u(v)=v\circ u\star u^{-1}.
    \end{equation*}
\end{proposition}
\begin{definition}
    An \emph{ideal} $I$ of a left wire $W$ is a normal subgroup of $W_\star$ such that $\lambda_u(I)\subset I$ and $(u\star I)\circ v\subset u\circ v \star I$ for all $u,v\in W$.

    An ideal $I$ of a right wire $W$ is a normal subgroup of $W_\star$ such that $\rho_v(I)\subset I$ and $u\circ(v\star I)\subset u\circ v \star I$ for all $u,v\in W$. 
\end{definition}
\begin{remark}
\label{rem:idealwire}
    Note that if $W$ is a wire, then an ideal of $W$ as a left wire is the same as an ideal of $W$ as a right wire. Moreover in this case, an ideal $I$ is a normal subgroup of $W$ such that $\lambda_u(I)\subset I$ and $\rho_u(I)\subset I$ for all $u\in W$. Or equivalently $u\circ I,I\circ u\subset u\star I=I\star u$ for all $u\in W$.
\end{remark}
\begin{definition}
    Let $f\colon W\to V$ be a wire morphism. We define $\ker(f)$ to be the set of elements $u\in W$ such that $f(u)=e$. 
\end{definition}
We present the last results of this section only for left wires, but they are also true for right wires.
\begin{proposition}
\label{pro:caraideal}
    Let $W$ be a left wire and $I\subset W$ such that $(I,\star)$ is a normal subgroup of $(W,\star)$. The following assertions are equivalent : \begin{enumerate}
        \item $I$ is an ideal of $W$.
        \item $I$ is the kernel of a wire morphism.
        \item The group quotient $W/I$ inherits a wire structure such that the canonical projection $W \to W/I$ is a wire morphism.
    \end{enumerate}
\end{proposition}
\begin{proof}
    The implications $(3)\Rightarrow (2)$ and $(2)\Rightarrow (1)$ are obtained by direct computations. 
    To obtain $(1)\Rightarrow (3)$ it is enough to show that the binary operation $[u]\circ[v]=[u\circ v]$ is well-defined on the group quotient $W/I$. Let $u,v\in W$ and $i,j\in I$, we have
    \[
        [(u\star i)\circ (v\star j)]= [(u\star i)\circ v \star \lambda_{u\star i}(j)]= [u\circ v].\qedhere
    \]
\end{proof}
\begin{proposition}
    Let $I$ be an ideal of a left wire $W$. Let $f\colon W\to V$ be a wire morphism such that $I\subset \ker(f)$. Then there exists a unique wire morphism $\bar{f}\colon W/I\to V$ such that the diagram
    \[\begin{tikzcd}
	{W/I} & V \\
	W
	\arrow[dashed,"{\exists !\bar{f}}", from=1-1, to=1-2]
	\arrow[from=2-1, to=1-1]
	\arrow["f"', from=2-1, to=1-2]
\end{tikzcd}\]
is commutative.
\end{proposition}
\begin{proof}
    The map $\hat{f}$ is defined by sending an element $[u]$ to $f(u)$; it is well-defined by hypothesis.
\end{proof}
\begin{definition}
    \label{def:comideal}
    Let $W$ be a wire. Denote by $[W,W]$ the usual commutator subgroup of $W_\star$. By Remark~\ref{rem:idealwire}, $[W,W]$ is an ideal of $W$.
\end{definition}
\begin{definition}
    Let $W$ be a left wire. We define $W^2$ to be the subgroup of $W_\star$ generated by the elements $u\cdot v=u^{-1}\star u\circ v\star v^{-1}$ and $W^2_{\mathrm{op}}$ to be the subgroup of $W_\star$ generated by the elements $u\cdot_{\mathrm{op}} v=u^{-1}\star v\circ u\star v^{-1}$
\end{definition}
\begin{proposition}
    Let $W$ be a left wire. The subgroups $W^2$ and $W^2_{\mathrm{op}}$ are ideals. The wires $W/W^2$ and $W/W^2_{\mathrm{op}}$ are respectively a trivial and an almost trivial skew brace.
    \end{proposition}
    \begin{proof}
To show that they are normal subgroups, one can use the identities \[(u\cdot v)^{w^{-1}}=(u\cdot w)^{-1}\star (u\cdot(w\star v))\quad \text{and}\quad (u\cdot_{\mathrm{op}}w)^{v}=((u\star v)\cdot_{\mathrm{op}}w)\star (v\cdot_{\mathrm{op}}w)^{-1}.\]
    Once we have shown that they are normal subgroups, it follows that $W^2$ and $W^2_{\mathrm{op}}$ are ideals by Proposition~\ref{pro:caraideal}, since in the quotient group $W / W^2$ we have $ [u \circ v] = [u \star v] $, and in $ W / W^2_{\mathrm{op}}$, we have  $[u \circ v] = [v \star u]$. 
\end{proof}
\begin{remark}
    One could also define $W^2$ and $W^2_{\mathrm{op}}$ in exactly the same way in a right wire $W$. A symmetric argument shows that in a right wire,  they are also ideals.
\end{remark}
\section{Free wires}
\label{sec:freewire}
     To construct the free wire on a set $X$ we will adapt the construction of $R(X)$ of Remark~\ref{rem:discfreecombr}, to a setting where \(\star\) and \(\circ\) are not commutative. The intuitive idea is to take the free semigroup $(M(X),\circ)$ generated by $X$ as the set of monomials. Then, we take the free group $(P(X),\star)$ generated by $M(X)$ as the set of polynomials. Now we want to extend the operation of the semigroup $M(X)$ to $P(X)$. A naive approach is to say that we extend it by brace distributivity, but this doesn't work.
    
    Indeed, let us look at the example where we try to define the circle operation between two words of length two, $(m_1\star m_2)\circ (m_3\star m_4)$. We could either distribute on the right and then on the left, which would give us \[m_1\circ m_3 \star m_3^{-1} \star m_2\circ m_3\star m_2^{-1}\star m_1^{-1}\star m_1\circ m_4\star m_4^{-1}\star m_2\circ m_4.\]
    Or we could do it the other way around. This time we obtain 
    \[m_1\circ m_3\star m_1^{-1}\star m_1\circ m_4 \star m_4^{-1}\star m_3^{-1}\star m_2\circ m_3\star m_2^{-1}\star m_2\circ m_4.\] 
    We see that the results are not the same. However they would coincide if the elements 
    \[m_3\cdot_{\mathrm{op}}m_2=m_3^{-1} \star m_2\circ m_3\star m_2^{-1}\quad \text{and}\quad m_1\cdot m_4=m_1^{-1}\star m_1\circ m_4\star m_4^{-1},
    \]
    commute. This motivates the following proposition, which is a direct generalization of a result of Trappeniers on two-sided skew braces~\cite{trappeniers2023two}.

    \begin{proposition}
    \label{pro:commutator trivial}
        Let $W$ be a wire. Then $[W^2,W_{\mathrm{op}}^2]=0$.
    \end{proposition}
    \begin{proof}
        It is now clear that $[u_1\cdot v_1,u_2\cdot_{\mathrm{op}}v_2]=e$ for all $u_1,u_2,v_1,v_2\in W$. The rest can be proved using the identities $[u,v]^{-1}=[v,u]$, and $[u,v\star w]=[u,v]\star [u,w]^{v^{-1}}$.
    \end{proof}
    \begin{notation}
        When we consider elements of a free group on a set of  generators $Y$ (for example $P(X)$ with $Y=M(X)$) we will denote them by \[\prod_{i=1}^nu_i^{\epsilon_i} = u_1^{\epsilon_1} \star u_2^{\epsilon_2} \star \dots \star u_n^{\epsilon_n},\] where $n$ is some non negative integer, the $u_i$ are elements of $Y$ and $\epsilon_i$ are elements of $\{1,-1\}$. We emphasize that the operation \( \star \) is non-commutative, and thus the order of the terms is essential. Note also that such an expression is not necessarily reduced.
    \end{notation}

    We argued that for two words $u,v\in P(X)$ of length greater than or equal to two, it is not possible to define a binary operation extending the one of $M(X)$, because of the multiple ways to distribute the $\circ$ operation over the $\star$ operation. However, it is possible when one of the words involved is a monomial, that is an element of $M(X)$. Indeed, Let $l\in M(X)$ and $u\in P(X)$ with $u=\prod_{i=1}^nu_i^{\epsilon_i}$. Then we can define $l\circ u=l\star\prod_{i=1}^n\lambda_l(u_i)^{\epsilon_i}$ and $u\circ l =\left(\prod_{i=1}^n\rho_l(u_i)^{\epsilon_i}\right)\star l$ where $\lambda_l(u_i)=l^{-1}\star l\circ u_i$ and $\rho_l(u_i)=u_i\circ l\star l^{-1}$. These formulas do not depend on the representative of $u$. Indeed, let $1\leq k<n$, and suppose that $u_{k+1}^{\epsilon_{k+1}}=u_k^{-\epsilon_k}$ then, we have
    \begin{align}
    \begin{split}
    \label{eq:welldefined}
    l\circ u&=l\star\left(\prod_{i=1}^{k-1}\lambda_l( u_i)^{\epsilon_i}\right)\star\lambda_l( u_k)^{\epsilon_k}\star\lambda_l(u_k)^{-\epsilon_k}\star\left(\prod_{i=k+1}^{n}(\lambda_l(u_i)^{\epsilon_i}\right)\\
    &=l\circ \left(\prod_{\substack{i=1\\i\neq k,k+1}}^nu_i^{\epsilon_i}\right).
    \end{split}
    \end{align}
    A symmetric computations shows that $u\circ l$ does not depend on the representative of $u$ either.
To extend this operation to general words, we have at least two choices.

    \begin{definition}
    \label{def:operations}
        Let $X$ be a set, for elements $u,v\in P(X)$ with $u=\prod_{i=1}^nu_i^{\epsilon_i}$ and $v=\prod_{i=1}^mv_i^{\delta_i}$ we define two operations 
        
        \begin{align*}
            &u\circ_r v=\left(\prod_{i=1}^n\rho_v(u_i)^{\epsilon_i}\right)\star v, \quad \text{and} \\
            &u\circ_l v= u\star\left(\prod_{i=1}^m\lambda_u(v_i)^{\delta_i}\right),
        \end{align*}
        where $\lambda_u(v_i)=u^{-1}\star u\circ v_i$ and $\rho_v(u_i)=u_i\circ v\star v^{-1}$.
    \end{definition}
    \begin{remark}
        A similar computation than in~\eqref{eq:welldefined} shows that the two operations do not depend on the word representatives.
    \end{remark}
    \begin{remark}
        The empty word $e$ can be written as an empty sum, so that \[u\circ_r e=e\circ_r u= e\circ_lu=u\circ_le=u.\]
    \end{remark}
The two operations $\circ_l$ and $\circ_r$ follow respectively left and right brace distributivity, i.e for all $u,v,w\in P(X)$,
\begin{align*}
    (u\star v)\circ_r w&=u\circ_r w\star w^{-1}\star v\circ_r w,\\
    u\circ_l(v\star w)&=u\circ_l v\star u^{-1}\star u\circ_l w.
\end{align*} 
\begin{proposition}
    \label{pro:associative}
    $P(X)^l=(P(X),\star,\circ_l)$ and $P(X)^r=(P(X),\star,\circ_r)$ are respectively left and right wires.
\end{proposition}
\begin{proof}
    It is left to see that the operations $\circ_l$ and $\circ_r$ are associative. We will only deal with $\circ_l$ as the argument to show that $\circ_r$ is associative is symmetric. Let $u,v,w\in P(X)$, suppose that $u=\prod_{i=1}^nu_i^{\epsilon_i}$, $v=\prod_{i=1}^mv_i^{\delta_i}$ and $w=\prod_{i=1}^lw_i^{\alpha_i}$. Then by the left brace distributivity,
    \[u\circ_l (v\circ_l w)= u\circ_lv\star \left(\prod_{i=1}^l \left((u\circ_l v)^{-1}\star u\circ_l (v\circ w_i)\right)^{\alpha_i}\right)\]
    and
    \[(u\circ_l v)\circ_l w= u\circ_lv\star \left(\prod_{i=1}^l \left((u\circ_l v)^{-1}\star(u\circ_l v)\circ w_i\right)^{\alpha_i}\right).\]
    Therefore, we may assume that $w\in M(X)$. Then,
    \[u\circ_l (v\circ w)=\left(\prod_{i=1}^m \left(u\circ_l (v_i\circ w)\star (u\circ w)^{-1}\right)^{\delta_i}\right)\star u\circ w\]
    and
    \begin{align*}
    (u\circ_l v)\circ w=&u\circ w\star\left(\prod_{i=1}^m \left((u\circ w)^{-1}\star(u\circ_l v_i)\circ w\right)^{\delta_i}\right)\\
    =&\left(\prod_{i=1}^m \left((u\circ_l v_i)\circ w\star (u\circ w)^{-1}\right)^{\delta_i}\right)\star u\circ w.
    \end{align*}
    Thus, we can suppose that $v\in M(X)$. Since $M(X)$ is associative we have, 
    \[
        u\circ (v\circ w) =(u\circ v)\circ w.\qedhere
    \]
\end{proof}
\begin{notation}
    We will denote the canonical actions (Definition~\ref{rem:actions}) of $P(X)^l$ by $\lambda$ and the one of $P(X)^r$ by $\rho$. Note that it is compatible with the notations introduced in Definition~\ref{def:operations}.
\end{notation}
\begin{lemma}
\label{lem:idealcrit}
    Let $I$ be a normal subgroup of $P(X)$. Then $I$ is an ideal of $P(X)^l$ if and only if $\lambda_u(I)\subset I$ and $\rho_s(I)\subset I$ for all $u\in P(X)$ and $s\in M(X)$. 
    
    Similarly, $I$ is an ideal of $P(X)^r$ if and only if $\rho_u(I)\subset I$ and $\lambda_s(I)\subset I$ for all $u\in P(X)$ and $s\in M(X)$
\end{lemma}
\begin{proof}
    We will prove the lemma only for $P(X)^l$. It is clear that if $I$ is an ideal of $P(X)^l$ then $\lambda_u(I)\subset I$ and $\rho_s(I)\subset I$ for all $u\in P(X)$ and $s\in M(X)$. Suppose the latter, it is left to show that $(u\star I)\circ_l v \subset u\circ_lv \star I$ for all $u,v\in P(X)$. Let $u=\prod_{i=1}^nu_i^{\epsilon_i}$ and $v=\prod_{i=1}^mv_i^{\delta_i}$ and $k\in I$. 
        \begin{align*}
            (u\star k)\circ_l v=& u\star k \star \prod_{i=1}^m \left((u\star k)^{-1}\star (u\star k)\circ v_i\right)^{\delta_j}\\
            =& u\star k \star \prod_{i=1}^m \left((u\star k)^{-1}\star u\circ v_i\star \rho_{v_i}(k)^{v_i}\right)^{\delta_j}\\
            \equiv & u\circ_l v \mod I.\qedhere
        \end{align*}
\end{proof}
\begin{corollary}
    If $I$ and $J$ are ideals of $P(X)^l$ (respectively $P(X)^r$), then $I\star J$ is an ideal of $P(X)^l$ (respectively $P(X)^r$).
\end{corollary}
\begin{proof}
    It is a direct consequence of Lemma~\ref{lem:idealcrit}.
\end{proof}
\begin{definition}
    Let $P(X)^2$ (respectively $P(X)^2_{\mathrm{op}}$) be the normal subgroup of $P(X)$ generated by the elements $m^{-1}\star m\circ l\star l^{-1}$ (respectively, $m^{-1}\star l\circ m\star l^{-1}$) for all monomials $m,l\in M(X)$. 
\end{definition}
\begin{lemma}
    \label{lem:staroperation}
    Let $u,v\in P(X)$. In $P(X)/P(X)^2$ (respectively $P(X)/P(X)^2_{\mathrm{op}}$), we have $[u\circ_r v]=[u\circ_l v]=[u\star v]$ (respectively, $[u\circ_r v]=[u\circ_l v]=[v\star u]$).
\end{lemma}
\begin{proof}
Suppose first that $v\in M(X)$. Then, the statement follows from the fact that for all $l\in M(X)$, in $P(X)/P(X)^2$ we have $[\rho_v(l)]= [\lambda_v(l)]=[l]$ and in $P(X)/P(X)_{\mathrm{op}}^2$ we obtain $[\rho_v(l)]=[l^{v^{-1}}]$ and $[\lambda_v(u_i)]=[u_i^v]$. The latter implies that the identities are still true when $v$ is a general word, which concludes the proof. 
\end{proof}
\begin{proposition}
    The ideals $(P(X)^l)^2$ and $(P(X)^r)^2$ (respectively $(P(X)_{\mathrm{op}}^l)^2$ and $(P(X)_{\mathrm{op}}^r)^2$) are equal to $P(X)^2$ (respectively $P(X)^2_{\mathrm{op}}$).
\end{proposition}
\begin{proof}
   By lemmas~\ref{lem:staroperation} and~\ref{lem:idealcrit}, it follows that $P(X)^2$ and $P(X)^2_{\mathrm{op}}$ are both ideals of $P(X)^l$ and $P(X)^r$. Which yields the claim.
\end{proof}
\begin{proposition}
    \label{pro:comideal}
    The normal subgroup $[P(X)^2,P(X)_{\mathrm{op}}^2]$ is an ideal of both $P(X)^l$ and $P(X)^r$.
\end{proposition}
\begin{proof}
    Just note that $P(X)^2$ and $P(X)_{\mathrm{op}}^2$ are both ideals of $P(X)^l$ and $P(X)^r$ and the conditions of the Criterion from Lemma~\ref{lem:idealcrit} are given by actions by endomorphism which yields the claim.
\end{proof}

Our next goal is to show that $\circ_l$ and $\circ_r$ coincide in the quotient 
\begin{equation}
\label{eq:freenotation}
    F_w(X)=P(X)/[P(X)^2,P(X)^2_{\mathrm{op}}].
\end{equation}

\begin{lemma}
    \label{lem:entanglement}
    There are actions by group automorphisms of the group \\
    $P(X)/P(X)^2$ on $(P(X)_{\mathrm{op}}^2/[P(X)^2,P(X)^2_{\mathrm{op}}],\star)$ and the group $(P(X)/P(X)_{\mathrm{op}}^2,\star)$ on $(P(X)^2/[P(X)^2,P(X)^2_{\mathrm{op}}],\star)$. These actions are given by the maps,
    \[P(X)/P(X)^2\to \Aut((P(X)_{\mathrm{op}}^2/[P(X)^2,P(X)^2_{\mathrm{op}}],\star));  [u]\mapsto([v]\mapsto [v^u])\]
    and 
    \[(P(X)/P(X)_{\mathrm{op}}^2,\star)\to \Aut((P(X)^2/[P(X)^2,P(X)^2_{\mathrm{op}}],\star));  [u]\mapsto([v]\mapsto [v^u]).\]  
\end{lemma}
\begin{proof}
    Let $z\in P(X)^2$ and $u,v\in P(X)$, such that $u\star v^{-1}\in P(X)^2_{\mathrm{op}}$, then $z^{u\star v^{-1}}z^{-1}\in [P(X)^2,P(X)^2_{\mathrm{op}}]$ so that $[z^u]\equiv [z^v]\mod [P(X)^2,P(X)^2_{\mathrm{op}}]$. Similarly, one sees that the second morphism is well-defined.
\end{proof}
\begin{remark}
    A consequence of Lemma~\ref{lem:entanglement} is that for all $u,v\in P(X)$, $z_1\in P(X)^2$ and $z_2\in P(X)_{\mathrm{op}}^2$ the identities $[z_1^{u\circ_l v}]=[z_1^{u\circ_r v}]=[z_1^{v\star u}]$ and $[z_2^{u\circ_l v}]=[z_2^{u\circ_r v}]=[z_2^{u\star v}]$ hold in $P(X)/[P(X)^2,P(X)_{\mathrm{op}}^2]$.
\end{remark}
\begin{proposition}
    \label{pro:coincide}
    Let $u,v\in P(X)$, then using the notation of Equation~\eqref{eq:freenotation}, $u\circ_r v$ and $u\circ_l v$ have the same image under the canonical homomorphism $P(X)\to F_w(X)$.
\end{proposition}
For the proof of Proposition~\ref{pro:coincide} we will proceed with an inductive argument for which we will need the following Lemma.
\begin{lemma}
\label{lem:recursion}
    Let $u,v\in P(X)$ such that $v=\prod_{i=1}^{m}v_i^{\delta_i}$ and $\delta_m=1$. Let $v_s=\prod_{i=1}^{m-1}v_i^{\delta_i}$. Then, $u\circ_rv\star u\circ_l v^{-1}\equiv u\circ_rv_s\star u\circ_lv_s^{-1}\mod [P(X)^2,P(X)^2_{\mathrm{op}}]$. 
\end{lemma}
\begin{proof}
    Notice that in the word $u\circ_r v\star u\circ_l v^{-1}$ each monomial $u_i\circ v_j^{\epsilon_i\delta_j}$ coming from $u\circ_rv$ has a corresponding opposite monomial $u_i\circ v_j^{-\epsilon_i\delta_j}$ coming from $u\circ_lv^{-1}$. We will see that when $j=m$, they conjugate an element of $P(X)^2_{\mathrm{op}}$. This way, we will be able to change the $\circ$ operation into a $\star$ by using Lemma~\ref{lem:entanglement} and conclude. Indeed, assuming the latter we obtain $\mod [P(X)^2,P(X)^2_{\mathrm{op}}]$,
    \begin{align*}
        u\circ_rv\star u\circ_l v^{-1}&\equiv \left(\prod_{i=1}^n\left(u_i\star\left(\prod_{j=1}^{m-1}\lambda_{u_i}(v_j)^{\delta_j}\right)\star u_i^{-1}\star u_i\star v_m\star v^{-1}\right)^{\epsilon_i}\right)\star v \\
        &\star( u^{-1}\star u\star v_m)^{-1}\star \left(u\star\prod_{j=1}^{m-1}\lambda_u(v_j)^{\delta_j}\right)^{-1}\\
        &\equiv u\circ_r v_s \star u\circ_l v_s^{-1} \mod [P(X)^2,P(X)^2_{\mathrm{op}}].
    \end{align*} 
    One can see that when $\epsilon_l=1$ the word which is conjugated by $u_l\circ v_m^{\epsilon_l}$ is
    \[
    v^{-1}\star\left(\prod_{i=l+1}^n\rho_v(u_i)^{\epsilon_i}\right)\star v\star v_m^{-1}\star \left(\prod_{i=l+1}^n\rho_{v_m}(u_i)^{\epsilon_i}\right)^{-1}\star v_m 
    \]
     and when $\epsilon_l=-1$,
    \[
    \left(\prod_{j=1}^{m-1}(u_i\circ v_j\star u_i^{-1})^{\delta_j}\right)^{-1}\star\left(\prod_{i=l+1}^n\rho_v(u_i)^{\epsilon_i}\right)\star v\star v_m^{-1}\star\left(\prod_{i=l+1}^n\rho_{v_m}(u_i)^{\epsilon_i}\right)^{-1}.
    \]
    In both cases the expression can be seen to be trivial in $P(X)/P(X)^2_{\mathrm{op}}$.
\end{proof}
\begin{proof}[Proof of Proposition~\ref{pro:coincide}]
    Let $u,v\in P(X)$, to show the congruence 
    \[u\circ_l v\equiv u\circ_r v\mod [P(X)^2,P(X)^2_{\mathrm{op}}],\] we will proceed by induction on the size of $v$.
    The statement is clear when $v$ is trivial.

    Suppose the proposition is true when $v$ has length less than or equal to $m$. Suppose $v= \prod_{i=1}^{m+1}v_j^{\delta_j}$, if $\delta_{m+1}=1$ we can apply Lemma~\ref{lem:recursion} and use the induction hypothesis. Otherwise we have by induction hypothesis, $0\equiv u\circ_l(v\star v_{m+1})\star u\circ_r(v\star v_{m+1})^{-1}$. Then we can conclude using Lemma~\ref{lem:recursion}.
\end{proof}

Following Propositions~\ref{pro:coincide} and~\ref{pro:comideal}, we can give $F_w(X)$ a wire structure with $[u]\circ [v] = [u\circ_l v]=[u\circ_r v]$.
\begin{proposition}
    $F_w(X)$ is the free wire generated by the set $X$. Explicitly, for each wire $W$ and map $f\colon X\to W$, there exists a unique wire morphism $\hat{f}\colon F_w(X)\to W$ such that $\hat{f}\iota= f$ where $\iota\colon X\to F_w(X)$ is the canonical injection.
\end{proposition}
\begin{proof}
    Let $W$ be a wire and $f\colon X\to W$ a map. Then there exists a unique semigroup morphism $f_s\colon M(X)\to W_\circ$ such that the following diagram is commutative.
    \[\begin{tikzcd}
	{M(X)} & W \\
	X
	\arrow[dashed,"{f_s}", from=1-1, to=1-2]
	\arrow["\iota", from=2-1, to=1-1]
	\arrow["f"', from=2-1, to=1-2]
\end{tikzcd}\]
Then, there exists a unique group morphism $f_g\colon P(X)\to W_\star$ such that the diagram
\[\begin{tikzcd}
	{P(X)} && W \\
	{M(X)} 
	\arrow[dashed,"{f_g}", from=1-1, to=1-3]
	\arrow[from=2-1, to=1-1]
	\arrow["{f_s}", curve={height=6pt}, from=2-1, to=1-3]
\end{tikzcd}\]
commutes.

By definition, $f_g$ is a wire morphism for both $\circ_l$ and $\circ_r$.
By Proposition~\ref{pro:commutator trivial}, $[W^2,W_{\mathrm{op}}^2]$ is trivial. As $f_g([P(X)^2,P(X)_{\mathrm{op}}^2])\subset [W^2,W_{\mathrm{op}}^2]$, the map $f_g$ extends uniquely to a wire morphism $\hat{f}\colon F_w(X)\to W$ such that the diagram
\[\begin{tikzcd}
	{F_w(X)} && W \\
	P(X)
	\arrow[dashed,"{\hat{f}}", from=1-1, to=1-3]
	\arrow[from=2-1, to=1-1]
	\arrow["{f_g}"', curve={height=6pt}, from=2-1, to=1-3]
\end{tikzcd}\]
commutes.
 The uniqueness of $\hat{f}$ follows from the uniqueness at each step.
\end{proof}
\begin{corollary}
    The quotient $P(X)/P(X)^2$ is a trivial skew brace, and it is the free group generated by $X$. The quotient $P(X)/P(X)^2_{\mathrm{op}}$ is the free almost trivial skew brace generated by $X$.
\end{corollary}
\begin{proof}
    It is enough to see that the groups $I=P(X)^2/[P(X)^2,P(X)^2_{\mathrm{op}}]$ and $I_{\mathrm{op}}=P(X)^2_{\mathrm{op}}/[P(X)^2,P(X)^2_{\mathrm{op}}]$ are ideals of $F_w(X)$ and that $P(X)/P(X)^2\cong F_w(X)/I$ and $P(X)/P(X)^2_{\mathrm{op}}\cong F_w(X)/I_{\mathrm{op}}$. Then the result follows from the universal property of the quotient.\qedhere

\end{proof}

\section{Free commutative skew braces}
\label{sec:freeskb}
\label{sect:freeabwire}
    To construct the free commutative skew brace on a set $X$, we first construct the free commutative wire on $X$.
    A simple way of constructing the latter is as a quotient of $F_w(X)$, simply asking the generator to commute with respect to the operation $\circ$.
    \begin{definition}
        We define the circle commutator $[P(X),P(X)]_\circ$ of $P(X)$ to be the normal subgroup of $(P(X),\star)$ generated by the elements \[(x_1\circ \dots \circ x_n)^{-1}\star x_{\sigma(1)}\circ\dots \circ x_{\sigma(n)}\]
        with $x_1,\dots, x_n\in X$ and $\sigma \in \Sym(n)$. Furthermore, we define the circle commutator of $F_w(X)$ to be 
        \[[F_w(X),F_w(X)]_\circ=\left([P(X),P(X)]_\circ\star [P(X)^2,P(X)_{\mathrm{op}}^2]\right)/ [P(X)^2,P(X)_{\mathrm{op}}^2].
        \]
    \end{definition}
    \begin{proposition}
        The normal subgroup $[P(X),P(X)]_\circ$ is an ideal of $P(X)^l$ and $P(X)^r$. In particular, $[F_w(X),F_w(X)]_\circ$ is an ideal of $F_w(X)$.
    \end{proposition}
    \begin{proof}
        We show the proposition only for $P(X)^l$ as the situations are symmetric. Let $m=x_1\circ \dots \circ x_n$, $m_\sigma=x_{\sigma(1)}\circ\dots \circ x_{\sigma(n)}$ and $u=\prod_{i=1}^n u_i^{\epsilon_i}$, We have
        \[
            \lambda_u(m^{-1}\star m_\sigma)=m^{-1}\star \prod_{i=n}^1 \rho_{m}(u_i)^{-\epsilon_i}\star \prod_{i=1}^n \rho_{m_\sigma}(u_i)^{\epsilon_i}\star m_\sigma.
        \]
        Since $\rho_m(u_i)^{-\epsilon_i}\star \rho_{m_\sigma}(u_i)^{\epsilon_i}\in [P(X),P(X)]_\circ$, we can deduce by induction that $\lambda_u(m^{-1}\star m_\sigma)\in [P(X),P(X)]_\circ$ so that $\lambda_u([P(X),P(X)]_\circ)\subset [P(X),P(X)]_\circ$. With a similar argument, on can see that $\rho_l([P(X),P(X)]_\circ)\subset [P(X),P(X)]_\circ$ for all monomial $l\in M(X)$. This concludes the proof by Lemma~\ref{lem:idealcrit}.\qedhere

    \end{proof}
    \begin{corollary}
       The wire 
       \[F_{\mathrm{cw}}(X)=F_w(X)/[F_w(X),F_w(X)]_\circ
       \]
       is the free commutative wire on the set $X$.
    \end{corollary}
    \begin{proof}
        This is a consequence of the universal property of quotients.
    \end{proof}
    Another more explicit construction can be done. We proceed as we did in the previous section but instead of taking as the set of monomials the free semigroup $M(X)$, we use the free commutative semigroup $(M_{\mathrm{c}}(X),\circ)$ generated by $X$. Then we take as the set of polynomials the free group $(P_{\mathrm{c}}(X),\star)$ generated by $M_{\mathrm{c}}(X)$. We can define on $P_{\mathrm{c}}(X)$ the operations $\circ_r$ and $\circ_l$ with the same formulas as in Definition~\ref{def:operations}. Notice that in this case  $u\circ_l v=v\circ_ru$. In fact, using the universal property of $P(X)$ and $P_{\mathrm{c}}(X)$ and the quotient, one can see that as groups, $P_{\mathrm{c}}(X)\cong P(X)/[P(X),P(X)]_\circ$ so that $P_{\mathrm{c}}(X)^l=(P_{\mathrm{c}}(X),\star,\circ_l)\cong P(X)^l/[P(X),P(X)]_\circ$ and $P_{\mathrm{c}}(X)^r=(P_{\mathrm{c}}(X),\star,\circ_r)\cong P(X)^r/[P(X),P(X)]_\circ$.

    Let $P_{\mathrm{c}}(X)^2$ be the normal subgroup of $P_{\mathrm{c}}(X)$ generated by the elements $m\cdot l=m^{-1}\star m\circ l\star l^{-1}$ for all $m,l\in M_{\mathrm{c}}(X)$. $P_{\mathrm{c}}(X)^2$ is the image of both $P(X)^2$ and $P(X)^2_{\mathrm{op}}$ under the canonical morphism $P(X)\to P_{\mathrm{c}}(X)$. Using standard arguments, one can show that the image of an ideal by a surjective wire morphism is an ideal. Hence, $P_{\mathrm{c}}(X)^2$ is an ideal of both $P_{\mathrm{c}}(X)^l$ and $P_{\mathrm{c}}(X)^r$. It follows that $[P_{\mathrm{c}}(X)^2,P_{\mathrm{c}}(X)^2]$ is also an ideal of both. In fact, there is an isomorphism as groups
    \[
        F_w(X)/[F_w(X),F_w(X)]_\circ\cong P_{\mathrm{c}}(X)/[P_{\mathrm{c}}(X)^2,P_{\mathrm{c}}(X)^2]
    \]
    which can be seen to be an isomorphism of left and right wires. In particular the left and right wire structures of $P_{\mathrm{c}}(X)/[P_{\mathrm{c}}(X)^2,P_{\mathrm{c}}(X)^2]$ are identical, meaning that it is a wire.
    
    \begin{proposition}
        The wire $P_{\mathrm{c}}(X)/[P_{\mathrm{c}}(X)^2,P_{\mathrm{c}}(X)^2]$ is the free commutative wire generated by $X$.
    \end{proposition}
    
   To understand the structure of this quotient, Fox calculus will be employed. We begin by briefly recalling the theory of derivations in group rings, as presented in~\cite{MR53938}. Basic properties of group rings are stated without proof; details can be found in Chapter 3, Section 3 of~\cite{milies2002introduction}.
    \begin{definition}
        Let $G$ be a group and $R$ a unital ring. The group ring of $G$ over $R$ is the set of formal sums
        \[\sum_{g\in G}a_gg
        \]
        such that $a_g\in R$ and almost all $a_g$ are zero. The sum is defined component wise and the multiplication is extended from $G$ by bilinearity. The unit $1$ of the multiplication is the formal sum whose coefficients are all zero except $a_e=1$. When $R$ is commutative, $R[G]$ is an $R$-algebra and it is called the group algebra of $G$ over $R$.
    \end{definition}
    \begin{proposition}
        Let $G$ and $H$ be groups. Any group morphism $\phi \colon G\to H$ can be extended to an algebra morphism $\phi^* \colon R[G]\to R[H]$ by defining $\phi^*\left(\sum_{g\in G}a_gg\right)=\sum_{g\in G}a_g\phi(g)$.
    \end{proposition}
    \begin{definition}
    \label{def:augmentationmap}
        The augmentation map $\alpha\colon R[G]\to R$ is the extension of the trivial morphism $G\to \{e\}$.
    \end{definition}
    \begin{definition}
        Let $G$ be a group and $H$ a subgroup of $G$. We define $\Delta(H,G)$ to be the left ideal of $R[G]$ generated by the elements $h-1$ with $h\in H$.
    \end{definition}
    \begin{proposition}
        \label{pro:normalideal}
        If $H$ is a normal subgroup of $G$, then $\Delta(H,G)$ is a two-sided ideal of $R[G]$ and $R[G]/\Delta(H,G)\cong R\left[G/H\right]$.
    \end{proposition}
        
    For the next few results we will denote by $F$ the free group generated by some set $Y$.
    \begin{definition}
        Let $u=\prod_{i=1}^nu_i^{\epsilon_i}\in F$ and $1\leq k\leq n$. The $k$-th initial section $u_{(k)}$ of $u$ is $\prod_{i=1}^{k-1}u_i^{\epsilon_i}$ if $\epsilon_{k}=1$ and $\prod_{i=1}^{k}u_i^{\epsilon_i}$ if $\epsilon_k=-1$. 
    \end{definition}
    \begin{definition}[(2.8) of~\cite{MR53938}]
    \label{def:partialderiv}
        Let $s\in Y$, we define the Fox derivative with respect to $s$ of a word $u=\prod_{i=1}^nu_i^{\epsilon_i}\in F$ to be
            \[\pdv{u}{s}=\sum_k\epsilon_ku_{(k)}\in \mathbb{Z}[F].\]
        where the sum runs over all indices $k$ such that $u_k=s$.
    \end{definition}
    Note that Definition~\ref{def:partialderiv} is not the original definition of the derivative as given in~\cite{MR53938}, but this is really how we think about derivatives when we do computations in the present paper.
    \begin{remark}
        The definition of the derivative does not depend on the representative of the word $u$.
    \end{remark}
    \begin{remark}
        One can extend the derivatives by linearity into maps $\pdv{}{s}\colon \mathbb{Z}[F]\to \mathbb{Z}[F]$.
    \end{remark}
    \begin{proposition}
        Let $u,v\in \mathbb{Z}[F]$ and $s\in Y$ then
        \[\pdv{u\cdot v}{s}=\pdv{u}{s}\alpha(v)+u\pdv{v}{s}.
        \]
        In particular if $u\in F$, 
        \[\pdv{u^{-1}}{s}=-[u^{-1}]\pdv{u}{s}.\]
    \end{proposition}
    The reason we are interested in Fox derivatives is that they provide a way to solve the word problem in $F/[N,N]$ where $N$ is a normal subgroup of $F$.
    \begin{proposition}[4.9 of~\cite{MR53938}]
        Let $N$ be a normal subgroup of $F$. An element $u\in F$ is in $[N,N]$ if and only if $\pdv{u}{s}\in \Delta(N,F)$ for all $s\in Y$.
    \end{proposition}

    Determining whether $\pdv{u}{s}\in \Delta(N,F)$ is not difficult when the quotient $F/N$ is easy to handle.
    We now consider the case 
$Y=M_{\mathrm{c}}(X)$, $F=P_{\mathrm{c}}(X)$ and $N=P_{\mathrm{c}}(X)^2$. We will denote $[N,N]$ by $N'$. In this case, $F/N$ is the free commutative group generated by $X$.
    \begin{proposition}
        \label{pro:indeprepre}
        Two words $u,v\in F$ are congruent modulo $N'$ if and only if $\pdv{u}{s}=\pdv{v}{s}$ modulo $\Delta(N, F)$ for all $s\in Y$.
    \end{proposition}
    \begin{proof}
        First notice that in $F/N$,  $[u]=\prod_{m\in Y}[m]^{\alpha\left(\pdv{u}{m}\right)}$ where $\alpha \colon\mathbb{Z}\left[F\right]\to \mathbb{Z}$ is the augmentation map. Thus, the fact that $\pdv{u}{s}=\pdv{v}{s}$ modulo $\Delta(N, F)$ for all $s\in Y$ implies that $[u]=[v]$ in $F/N$. Hence, the claim follows from the identity $\pdv{u\star v^{-1}}{s}=\pdv{u}{s}-(u\star v^{-1})\pdv{v}{s}$.
    \end{proof}
        \begin{notation}
        \label{not:notationsimple}
            In light of Proposition~\ref{pro:indeprepre}, we will, from now on, work only with the residues of the Fox derivatives in \(\mathbb{Z}[F/N]\). To simplify notation, we omit the brackets and write \(\pdv{u}{s}\) instead of \(\left[\pdv{u}{s}\right]\). Since we will no longer consider these elements in \(\mathbb{Z}[F]\), this should not cause confusion.

            Furthermore,  note that because of Proposition~\ref{pro:indeprepre} Fox derivatives factorize through maps $ \mathbb{Z}\left[F/N'\right]\to \mathbb{Z}\left[F/N\right]$. In the sequel, we will sometimes abuse notation and write $\pdv{u}{s}$ for an element $u\in \mathbb{Z}\left[F/N'\right]$, meaning that it is the residue of the Fox derivative of any representative of $u$ in $\mathbb{Z}\left[F/N\right]$.
        \end{notation}
        The following definition is a slight modification of the Magnus embedding~\cite{a3e22350-c9b8-38aa-8153-d26817696fd2}.
    \begin{definition}
    \label{def:asspoly}
        We can associate a polynomial $f_u$ of $\mathbb{Z}[F/N][(t_x)_{x\in X}]$ to any element $u$ of $F$ such that \[f_u=\left(1-\sum_{s\in Y}\pdv{u}{s}\right)+\sum_{s\in Y} \pdv{u}{s}t_s.\]
        By $t_s$ we mean that if $s=x_1\circ \dots \circ x_i$, then $t_s=t_{x_1}\dots t_{x_i}$.
    \end{definition}
   The image of the map
\begin{equation}
\label{eq:deevaluation}
\mathfrak{d}\colon F/N' \to \mathbb{Z}[F/N][(t_x)_{x \in X}], \quad [u] \mapsto f_u
\end{equation}
can be described in terms of evaluation maps. It is immediate that this image lies in \(\eva_1^{-1}(1)\), where
\[
\eva_1\colon \mathbb{Z}[F/N][(t_x)_{x \in X}] \to \mathbb{Z}[F/N], \quad t_x \mapsto 1.
\]
It is also contained in \(\eva_X^{-1}(F/N)\), where
\[
\eva_X\colon \mathbb{Z}[F/N][(t_x)_{x \in X}] \to \mathbb{Z}[F/N], \quad t_x \mapsto [x],
\]
and \(F/N\) is identified with its image in \(\mathbb{Z}[F/N]\).  In fact, we have 
            \begin{equation}
            \label{eq:idk}
                \eva_{X}(f_u)=[u]
            \end{equation}
            for all $u\in F$. Indeed, let $s\in Y$, then $f_s=t_s$ so \eqref{eq:idk} is verified for every generators of $F$. Let $u,v\in F$ for which \eqref{eq:idk} is true, then 
            \begin{align*}
                1+\sum_{s\in Y}\pdv{u\star v^{-1}}{s}([s]-1)([s]-1)=&1+\sum_{s\in Y} \left(\pdv{u}{s}-[u\star v^{-1}]\pdv{v}{s}\right)([s]-1)\\
                =&1+[u]-1-[u\star v^{-1}]([v]-1)\\
                =& [u\star v^{-1}].
            \end{align*}
            So it holds for every elements of $F$.

            Recall from Example~\ref{ex:wiresringmorph} that $\eva_{1}^{-1}(1)\cap \eva_{X}^{-1}(F/N)$ can be given a structure of commutative wire with
        \[
        f\star g= f+\eva_X(f)(g-1)\quad \text{and}\quad f\circ g=f\cdot g
        \]
        for all $f,g\in \eva_1^{-1}(1)\cap \eva_X^{-1}(F/N)$. From now on, we will call this wire $\mathcal{W}(X)$. It is a direct consequence of~\eqref{eq:idk} that $\mathfrak{d}\colon(F/N',\star)\to (\mathcal{W}(X),\star)$ is a group morphism.
        \begin{thm}
    \label{pro:derivativeproduct}
        The map
        \[
        \mathfrak{d}\colon F/N' \to \mathcal{W}(X), \quad [u] \mapsto f_u
        \]
        is a wire morphism.
    \end{thm}
    \begin{proof}
    As mentioned in Notation~\ref{not:notationsimple}, when we write Fox derivatives, we consider in fact their image in $\mathbb{Z}[F/N]$.
        Let $u,v\in F$, we only need to show that,
        $
        f_{[u\circ_l v]}=f_uf_v.
        $
        It is equivalent to show that for all $s\in Y$,
        \begin{equation}
        \label{eq:pdvmorphism}
        \pdv{u\circ_l v}{s} =f_u(0)\pdv{v}{s}+f_v(0)\pdv{u}{s}+\sum_{\substack{s_1,s_2\\s_1\circ s_2=s}}\pdv{u}{s_1}\pdv{v}{s_2},
        \end{equation}
        where by $f_u(0)$ and $f_v(0)$ we mean the evaluation $t_x\mapsto 0$ for all $x\in X$.

        First we will show by induction that the proposition is true when $v\in Y$. The result is clear when either $u=e$ or $v=e$. Suppose that $v\in Y$. Let $u\in F$, such that \eqref{eq:pdvmorphism} is true for $u\circ v$, that is
        \[\pdv{u\circ v}{s} =f_u(0)\pdv{v}{s}+\sum_{\substack{s_1\\s_1\circ v=s}}\pdv{u}{s_1}.\]
        
        Let $u_1\in Y$ and $\epsilon\in \{1,-1\}$. Then,
        \begin{align*}
            \pdv{(u\star u_1^\epsilon)\circ v}{s}=&\pdv{u\circ v}{s}+\epsilon \left[u_1^{\frac{\epsilon -1}{2}}\star u\right]\left(-\pdv{v}{s}+
\pdv{u_1\circ v}{s}\right)\\
        =&\left(f_u(0)-\epsilon \left[u_1^{\frac{\epsilon -1}{2}}\star u\right]\right)\pdv{v}{s}+\sum_{\substack{s_1\\s_1\circ v=s}}\pdv{u}{s_1}+\epsilon \left[u_1^{\frac{\epsilon -1}{2}}\star u\right]
\pdv{u_1\circ v}{s}.
        \end{align*} 
        We have that $f_{u\star u_1^\epsilon}(0)= f_{u}(0)-\epsilon\left[u_1^{\frac{\epsilon-1}{2}}\star u\right]$. Moreover,
        \[
        \sum_{\substack{s_1\\s_1\circ v=s}}\pdv{u\star u_1^{\epsilon}}{s_1}= \sum_{\substack{s_1\\s_1\circ v=s}}\pdv{u}{s_1}+\epsilon \left[u_1^{\frac{\epsilon -1}{2}}\star u\right]\sum_{\substack{s_1\\s_1\circ v=s}}\pdv{u_1}{s_1}
        \]
        and
        \[
        \pdv{u_1\circ v}{s}= \sum_{\substack{s_1\\s_1\circ v=s}}\pdv{u_1}{s_1}.
        \]
        Hence the formula~\eqref{eq:pdvmorphism} is true for $(u\star u_1^\epsilon)\circ v$.
        We have shown that \eqref{eq:pdvmorphism} holds when $v\in Y$. Suppose now that $v=v_1^{-1}$ for some $v_1\in Y$, we have 
        \begin{align*}
            \pdv{u\circ_lv}{s}=&\pdv{u}{s}+[v]\left(-\pdv{u\circ v_1}{s}+\pdv{u}{s}\right)\\
            =&-[v]f_u(0)\pdv{v_1}{s}+([v]+[e])\pdv{u}{s}-[v]\sum_{\substack{s_1\\s_1\circ v_1=s}}\pdv{u}{s_1}            
        \end{align*}
        which is the formula~\eqref{eq:pdvmorphism}.
        
        Suppose now that $v=v_1\star v_2$ with $v_1,v_2\in F$, such that $u\circ_l v_1$ and $u\circ_l v_2$ verify~\eqref{eq:pdvmorphism}. Then,
        \[\pdv{u\circ_l v}{s}=\pdv{u\circ_lv_1}{s}-[v_1]\pdv{u}{s}+[v_1]\pdv{u\circ_lv_2}{s}.\]
        In addition, \[\sum_{\substack{s_1,s_2\\s_1\circ s_2=s}}\pdv{u}{s_1}\pdv{v_1}{s_2}+[v_1]\sum_{\substack{s_1,s_2\\s_1\circ s_2=s}}\pdv{u}{s_1}\pdv{v_2}{s_2}=\sum_{\substack{s_1,s_2\\s_1\circ s_2=s}}\pdv{u}{s_1}\pdv{v}{s_2}\]
        and $f_{v}(0)= f_{v_1}(0)+[v_1]f_{v_2}(0)-[v_1]$ and $f_u(0)\pdv{u\circ_l v}{s}=f_u(0)(\pdv{v_1}{s}+[v_1]\pdv{v_2}{s})$. So $u\circ_l v$ verifies~\eqref{eq:pdvmorphism}.
        \end{proof}
        \begin{lemma}
        \label{lem:quotientmodule}
        the commutative group \( (N/N', \star) \) is a \( \mathbb{Z}[F/N] \)-module via the action defined by
\[
 {[u]_N}\cdot[n]_{N'} = \left[n^{u^{-1}}\right]_{N'}
\]
for all $u \in F $ and $n \in N $, where $[u]_N \in F/N$ and $[n]_{N'} \in N/N'$ denote the equivalence classes of $u$ and $n$.

In addition, $\pdv{f\cdot v}{s}=f\pdv{v}{s}$ for all $f\in \mathbb{Z}[F/N]$ and $v\in N/N'$.
    \end{lemma}
    \begin{proof}
    Straightforward computations show that the action is well-defined. The second statement holds directly when $f \in F/N$. In the general case, the result follows from the fact that
\[
\frac{\partial (n_1 \star n_2)}{\partial s} = \frac{\partial n_1}{\partial s} + \frac{\partial n_2}{\partial s}
\quad \text{for all } n_1, n_2 \in N \text{ and } s \in Y.\qedhere
\]
    \end{proof}
    The following result and proof are just reformulation of Theorem 2 of~\cite{remeslennikov1970some}.
        \begin{thm}
        \label{thm:isomfreeabwire}
            The wire morphism $\mathfrak{d}$ of Theorem~\ref{pro:derivativeproduct} is an isomorphism.
        \end{thm}
        \begin{proof} 
            Let $f\in \mathcal{W}(X)$ and $v_1\in F$ be any representative of $v=\eva_{X}(f)^{-1}$. Then, $\eva_{X}(f\star f_{v_1})=1$. Since $\mathfrak{d}$ is a group morphism, we may assume that $v=1$. Because $f\in \eva_1^{-1}(1)$, we can write $f=1+\sum_{s\in Y}\gamma_s(t_s-1)$ for some $\gamma_s\in \mathbb{Z}[F/N]$. Then, in $\mathbb{Z}[F/N]$ we have
            \[\sum_{s\in Y}\gamma_s([s]-1)=0.\]
            
            Let $\beta_s\in \mathbb{Z}[F/N']$ be such that $[\beta_s]=\gamma_s$ in $\mathbb{Z}[F/N]$ and $w=\sum_s\beta_s([s]-1)\in \mathbb{Z}[F/N']$, notice that $w$ is such that $\pdv{w}{s}=\gamma_s$ for all $s\in Y$. Also by hypothesis, the image of $w$ is $0$ in $Z\left[F/N\right]$. Hence by Proposition~\ref{pro:normalideal} there exist elements $l_i\in \mathbb{Z}\left[F/N'\right]$ and $w_i\in N/N'$ such that 
            \[
            w=\sum_i l_i(w_i-1).
            \]
            Moreover,
            \[
                \pdv{w}{s}= \sum_i [l_i]\pdv{w_i}{s}.
            \]
            Then, using the action of Lemma~\ref{lem:quotientmodule}, $u=\prod_i [l_i]\cdot w_i\in N/N'$ is such that $\pdv{u}{s}=\pdv{w}{s}=\gamma_s$ for all $s\in Y$. Therefore, $\mathfrak{d}(u)=f$.
        \end{proof}
        \begin{remark}
            It is natural to expect that taking the quotient of \( F/N' \) by the commutator ideal \( [F/N', F/N'] \), as defined in Definition~\ref{def:comideal}, yields the free wire in which both operations are commutative. This can be made explicit via the isomorphism \( \mathfrak{d} \).

Let \( R((t_x)_{x\in X}) \subseteq \mathbb{Z}[(t_x)_{x\in X}] \) be the set of polynomials whose coefficients sum to one, the triple \( (R((t_x)_{x\in X}), \oplus, \cdot) \) is the free wire with both operations commutative generated by $X$ (see Remark~\ref{rem:discfreecombr} and Remark~\ref{rem:ababfreewire}). Define the wire morphism
\[
\bar{\alpha} \colon \mathcal{W}(X) \to R((t_x)_{x\in X}), \quad \gamma + \sum_{s \in Y} \gamma_s t_s \quad \mapsto \quad \alpha(\gamma) + \sum_{s \in Y} \alpha(\gamma_s) t_s
\]
as an extension of the usual augmentation map $\alpha$ of Definition~\ref{def:augmentationmap}.
The morphism \( \bar{\alpha} \) is surjective, since for any
\[ a + \sum_{s \in Y} a_s t_s \in R((t_x)_{x\in X}),
\]
a preimage under \( \bar{\alpha} \) is \( \mathfrak{d}\left([\prod_{s \in Y} s^{a_s}]\right) \) (here it doesn't matter in which order the monomials appear).

If $u \in F$, then $f_u \in \ker(\bar{\alpha})$ if and only if each monomial $m \in Y$ appears in the expression of $u$ the same number of times as $m^{-1}$. This condition is equivalent to $u \in [F, F]$, which implies that
\[
\ker(\bar{\alpha}) = \mathfrak{d}([F/N', F/N']).
\]
\end{remark}
    
        It is well known that if \( F/N \) is a torsion-free commutative group, then the ring \( \mathbb{Z}[F/N][(t_x)_{x\in X}] \) is an integral domain. To see this, it suffices to show that \( \mathbb{Z}[F/N] \) itself is an integral domain. Levi showed in~\cite{levi1913arithmetische} that every torsion-free commutative group admits a total order that is compatible with the group operation. That is, there exists a total order \( \leq \) on \( F/N \) such that for all \( x, y, z \in F/N \), if \( x \leq y \), then \( x \star z \leq y \star z \). For $F/N$ the free commutative group generated by $X$, we can simply take a lexicographic order.

Using this order, one can define the initial monomial of an element 
\[
f = \sum_{g\in F/N} \gamma_g g \in \mathbb{Z}[F/N]
\]
as
\[
\operatorname{in}(f) = \max\{ g \in F/N : \gamma_g \neq 0 \}.
\]
The properties of the ordering ensures that for all \( f, g \in \mathbb{Z}[F/N] \), we have
\[
\operatorname{in}(fg) = \operatorname{in}(f) \star \operatorname{in}(g),
\]
which implies that \( \mathbb{Z}[F/N] \) is an integral domain.

Consequently, the monoid \( (\mathbb{Z}[F/N][(t_x)_{x \in X}], \cdot) \) is cancellative, since it arises from multiplication in an integral domain. It follows that \( (F/N',\circ) \) is itself cancellative, so the commutative wire $(F/N',\star,\circ)$ embeds into its skew brace of fractions.

        \begin{thm}
            \label{theorem:freeskb}
            The free commutative skew brace generated by the set $X$ can be described as the set of fractions $\frac{u}{v}$ such that $u,v\in F/N'$ with relation $\frac{u}{v}=\frac{u_1}{v_1}$ if and only if $u\circ v_1=v \circ u_1$ and with operations 
            
            \[
                \frac{u}{v}\star \frac{u_1}{v_1}=\frac{u\circ v_1\star (v\circ v_1)^{-1}\star v\circ u_1}{v\circ v_1}
            \]
            and \[
                \frac{u}{v}\circ \frac{u_1}{v_1}=\frac{u\circ u_1}{v\circ v_1}.
            \]
        \end{thm}
        \begin{proof}
            This amounts to saying that the free commutative skew brace generated by the set $X$ is simply the skew brace of fractions of $F/N'$. This is the case by the universal property of the skew brace of fractions, see Proposition~\ref{pro:universalproperty}.
        \end{proof}
        \subsection*{Acknowledgements}
        This work was partially supported by
the project OZR3762 of Vrije Universiteit Brussel, the 
FWO Senior Research Project G004124N,
        and the author is supported by an ASP grant from FNRS. 

        I am very grateful to Victoria Lebed, with whom we started thinking about free skew braces in general. Her insightful comments significantly improved the quality of this paper. I also thank my supervisors, Leandro Vendramin and Joost Vercruysse, for their careful reading and valuable feedback.
        Finally, I would like to thank Be'eri Greenfeld, Eric Jespers, Silvia Properzi, Marco Trombetti and Arne Van Antwerpen for the interesting discussions on free skew braces in Banff, at the “Skew Braces, Braids and the Yang-Baxter Equation” workshop, and elsewhere.
    \bibliographystyle{abbrv}
\bibliography{ref}
\end{document}